\documentclass[12pt]{amsart}
\usepackage{epsfig}

\usepackage{amscd}
\usepackage[mathscr]{eucal}
\usepackage{amssymb}
\usepackage{amsxtra}
\usepackage{amsmath}
\usepackage[all]{xy}

\newcommand{\Qed}{\hfill \qedsymbol \medskip}

\newtheorem{thm}{Theorem}
\newtheorem*{rem}{Remark}
\newtheorem{lm}[thm]{Lemma}
\newtheorem{prop}[thm]{Proposition}

\newtheorem{cor}[thm]{Corollary}

\def\dd{\text{d}}

\def\Vol{\operatorname{Area}}

\def\RR{\mathbb{R}}
\def\ZZ{\mathbb{Z}}

\def\Ham{Ham}
\def\Cal{Cal}
\def\Eq{\mathcal{E}}
\def\Eqr{\hat{\mathcal{E}}}
\def\Disk{D}

\def\depth{\operatorname{depth}}

\begin{document}

\title{Hofer's metric on the space of diameters}

\thanks{The author was partially supported by the ISRAEL SCIENCE
FOUNDATION (grant No. 1227/06 *); This is a part of the author's PhD thesis,
being carried out under the guidance of Prof. P.~Biran, at Tel-Aviv
University.}

\date{\today}
\author{Michael Khanevsky}
\address{Michael Khanevsky, School of Mathematical Sciences, Tel-Aviv
  University, Ramat-Aviv, Tel-Aviv 69978, Israel}
\email{khanev@post.tau.ac.il}
\maketitle

\section{Main result} \label{S:result}

Let $\Disk \subset \RR^2$ be the open unit disk, endowed with the symplectic structure 
$\omega = \frac{1}{\pi} \dd x \wedge \dd y$ (so that $\Vol(\Disk) = \int_\Disk \omega = 1$).
A \emph{diameter} $L$ in $\Disk$ is the image of a smooth curve $\gamma: (-1, 1) \to \Disk$
for which there exists $\varepsilon > 0$ such that:
\[
	\gamma (t) = (t, 0) \quad \forall t \in (-1, -1 + \varepsilon] \cup [1-\varepsilon, 1),
\]
and which divides $\Disk$ into two components of equal area.
Denote by $\Eq$ the space of diameters in $\Disk$, endowed with the Hofer metric $d$
(see the definition below).
The present paper is dedicated to the study of metric properties of $(\Eq, d)$ and
their relation to Lagrangian intersections.

Our first result is:
\begin{thm} \label{T:infinite}
	The metric space $(\Eq, d)$ is unbounded: $\sup_{x, y} d(x, y) = \infty$.
\end{thm}

In the context of Lagrangian intersections we prove:
\begin{thm} \label{T:bound}
	Let $L, L' \subset \Disk$ be two diameters intersecting transversely at $\#(L \cap L')$ 
	points (see Section~\ref{S:bound} for the precise definition).
	Then $d(L, L') \leq \frac{1}{8} \cdot \# (L \cap L') + 1$.
	Moreover, this estimate is sharp in the sense that the linear bound cannot be improved.
	(However, the constant $\frac{1}{8}$ might not be the optimal one.)
\end{thm}

\medskip

The rest of the paper is organized as follows. In Section~\ref{S:infinite} we prove 
Theorem~\ref{T:infinite}. Section~\ref{S:bound} is devoted to the proof of Theorem~\ref{T:bound}. 
Finally, in Section~\ref{S:generalize} we discuss possible generalizations of our results to
other surfaces than the disk.

\medskip

\emph{Acknowledgements}.
I would like to thank prof. L.~Polterovich for suggesting me the proof of Theorem~\ref{T:infinite} 
using the theory of Calabi quasimorphisms as well as for fruitful discussions. I also thank prof. 
P.~Biran for his help with preparation of this paper.

\section{Infiniteness of the diameter} \label{S:infinite}

The distance between two diameters $L_1, \, L_2 \subset \Disk$ is defined by
\[
	d(L_1, L_2) = \inf_{\begin{subarray}{l} \phi \in \Ham(\Disk) \\ \phi (L_1) = L_2\end{subarray} } \| \phi \| ,
\]
where $\| \cdot \|$ stands for the Hofer norm on the group $\Ham(\Disk)$ of compactly supported Hamiltonian 
diffeomorhisms: 
\[
	\| \phi \| = \inf \int_0^1 \max_{p \in \Disk} H (p, t) - 
																								\min_{p \in \Disk} H (p, t) \mathrm{d}t ,
\]
where the infimum goes over all compactly supported Hamiltonians $H: \Disk \times [0, 1] \to \RR$ 
such that $\phi$ is the time-1 map of the corresponding flow.

The distance $d$ is well defined as $\Ham (\Disk)$ acts transitively on the space of diameters. 
Properties of Hofer metric on $\Ham$ imply that $d(\cdot, \cdot)$ is indeed a 
metric (for details, see ~\cite{C}) and is invariant under Hamiltonian diffeomorphisms:
\[
	d(L_1, L_2) = d(\phi L_1, \phi L_2), \quad \forall \phi \in \Ham(\Disk).
\]

\medskip

Denote by $L_0 = \{y = 0\} \subset \Disk$ the standard diameter. Denote by $S$ the stabilizer 
of $L_0$ in $\Ham(\Disk)$, that is,
\[
	S = \{ \phi \in \Ham (\Disk) \; | \; \phi (L_0) = L_0\} ,
\]
Transitivity of the action of $\Ham$ on $\Eq$ implies the identification 
$\Eq \simeq \Ham (\Disk) / S$ with the left cosets of $S$.
We define the \emph{reduced diameter space} to be
\[
	\Eqr = S \backslash \Eq = \{ L \in \Eq\} / \sim ,
\]
where $L \sim \phi L$ for any $\phi \in S$. An equivalent definition is 
\[
	\Eqr = ``S \backslash \Ham (\Disk) / S" = \{ h \in \Ham (\Disk)\} / \sim ,
\]
with $h \sim \phi h \psi$ for any $\phi, \psi \in S$.

\medskip
As $d (L_0, \phi L) = d (\phi L_0, \phi L) = d (L_0, L)$
for all $\phi \in S$, the distance from $L_0$ descends to $\Eqr$, hence 
is well defined on the reduced diameter space.

Theorem~\ref{T:infinite} follows from the following:
\begin{prop} \label{p:infinite}
	$\sup_L d(L_0, [L]) = \infty$, therefore the reduced space $\Eqr$ is unbounded.
\end{prop}

In the proof we present an explicit autonomous Hamiltonian flow $\phi$ which 
deforms the standard diameter arbitrary far from $L_0$ as $t \to \infty$.

\bigskip

The proof uses results from Entov-Polterovich ~\cite{E-P}. We start with a brief recollection of 
some relevant facts from that paper.

Let $F_t : \Disk \to \RR$, $t \in [0, 1]$ be a time-dependent smooth function with compact support. We define
$\widetilde{\Cal} (F_t) = \int_0^1 \left( \int_\Disk F_t \omega \right) \dd t$. As $\omega$ is exact on $\Disk$,
$\widetilde{\Cal}$ descends to a homomorphism $\Cal: \Ham(\Disk) \to \RR$ which is called 
the Calabi homomorphism. Clearly, for $\phi \in \Ham(\Disk)$, $\Cal(\phi) \leq \| \phi \|$.

Let $G$ be a group. A function $r : G \to \RR$ is called a \emph{quasimorphism} if there exists 
a constant $R$ such that $|r(fg) - r(f) - r(g)| < R$ for all $f, g \in G$. The quasimorphism $r$ 
is called \emph{homogeneous} if it satisfies $r(g^m) = mr(g)$ for all $g \in G$ and $m \in \ZZ$.
Any homogeneous quasimorphism satisfies $r(fg) = r(f) + r(g)$ for commuting elements $f, g$.

For a compactly supported function $F: \Disk \to \RR$ the Reeb graph $T_F$ is defined as a set of connected
componets of level sets of $F$ (for a detailed definition we refer the reader to ~\cite{E-P}). 
For a generic Morse function $F$ this set, equipped with the topology 
induced by the projection $\pi_F: \Disk \to T_F$, is homeomorphic to 
a tree. We endow $T_F$ with a probability measure given by $\mu (A) = \int_{\pi_F^{-1}(A)} \omega$ 
for any $A \subseteq T_F$ with measureable $\pi_F^{-1}(A)$. $\pi_F (\partial \Disk)$ will be 
referred to as the root of $T_F$.

The results from ~\cite{E-P} imply existance a family of homogeneous quasimorphisms 
$\{Cal_A\}_{A \in [1/2, 1)}$ on $\Ham(\Disk)$ with the following properties:
\begin{itemize}
	\item
		$Cal_A (\phi) = \Cal(\phi)$ for any $\phi$ supported in a disk of area $A$.
	\item
		$\Cal_A(\phi) \leq 2A \cdot \| \phi \|$.
	\item
		For $\phi \in \Ham(\Disk)$ generated by an autonomous function $F: \Disk \to \RR$, 
		$\Cal_A (\phi)$ can be computed in the the following way. Consider the set $I$ of all 
		points on the Reeb graph $T_F$ of $F$ which cut from the root subtrees of	total measure 
		at least $A$. $I$ is either a segment on $T_F$ which starts from the root or empty. 
		For non-empty $I$ denote by $x$ the other end of the segment, for $I = \emptyset$ 
		we set $x$ to be the root. Then $\Cal_A(\phi) = \Cal(\phi) - 2A \cdot F(x)$.
\end{itemize}

\medskip

\emph{We now prove Proposition~\ref{p:infinite}.} Let
\[
	r_A = \Cal_A - \Cal, \qquad A \in [1/2, 1) .
\]
For each $A$, $r_A$ is a homogeneous quasimorphism which is Lipschitz with respect to the 
Hofer distance on $\Ham(\Disk)$:
\begin{eqnarray*}
	| r_A (\phi) | & = & \left| \Cal_A (\phi) - \Cal (\phi) \right| \leq \\ 
	& \leq & 2A \cdot \|\phi\| + \|\phi\| \leq (1+2A) \|\phi\|
\end{eqnarray*}

We first show that $r_A$ vanishes on the stabilizer $S$.
Let $\phi \in \Ham (M)$ be such that $\phi (L_0) = L_0$. Replacing $\phi$ with $\psi \circ \phi$
where $\psi \in \Ham(\Disk)$ has arbitrary small norm we may assume that $\phi = \text{Id}$
in some neighborhood of $L_0$.
From the transitivity of $\Ham$ on each half-disk it follows that $\phi$ can be decomposed as
$\phi = \phi_n \circ \phi_s$ where $\phi_n$ ($\phi_s$) are Hamiltonians
supported in the upper (lower) half of $\Disk$ with respect to $L_0$. Since $\phi_n, \phi_s$ 
commute, we have
\[
	\Cal (\phi) = \Cal (\phi_n) + \Cal (\phi_s) = \Cal_A (\phi_n) + \Cal_A (\phi_s) =
\Cal_A (\phi)
\] 
hence $r_A (\phi) = 0$.
	
It follows that $r_A$ ``descends'' to $\Eqr$ in the following sense: for any $\phi \in \Ham (\Disk)$, 
$r_A (\phi)$ can be computed from $[\phi] \in \Eqr$ up to an error $C$ which is bounded by 
twice the defect of the quasimorphism. Moreover, we also have
\[
	d(\phi (L_0), L_0) \geq \frac{| r_A (\phi) | - C}{1+2A}.
\]
Assume that there exists a diffeomorphism $\phi \in \Ham(\Disk)$ for which 
$r_A (\phi) \neq 0$ for some $A \in [1/2, 1)$. Then 
\[
	d(\phi^n (L_0), L_0) \geq \frac{| r_A (\phi^n) | - C}{1+2A} = \frac{n \, | r_A (\phi) | - C}{1+2A}
	\xrightarrow[n \to \infty]{} \infty
\] 
Proposition~\ref{p:infinite} would follow if we show existence of such a $\phi$.

\medskip

For this end, let $H_\varepsilon: \Disk \to \RR$ be an autonomous smooth function given by 
$H_\varepsilon (r, \theta) = -\frac{r^2}{2} + \frac{1}{2}$ on the disk $\{ r \leq 1 - \varepsilon \}$, and which
equals zero near $\partial \Disk$. Let $\phi$ be its time-1 map. Then
the Reeb graph of $H_\varepsilon$ is a segment whose points correspond to circles around the origin 
(level sets of $H_\varepsilon$). The level set $x$ relevant for calculation of $Cal_A(\phi)$ is the 
circle given by $r_x^2 = A$ and so:
\[
	\Cal_A(\phi) = \Cal(\phi) - 2A \cdot H_\varepsilon(x) = \Cal(\phi) + 2A \cdot \frac{A - 1}{2}.
\]
It follows that
\[
 	r_A (\phi) = \Cal_A (\phi) - \Cal (\phi) = A (A-1) \neq 0.
\]

This completes the proof of Proposition~\ref{p:infinite}, hence of Theorem~\ref{T:infinite}, too.

\Qed

In the following section we prove that for any diameter $L$ which is transverse to $L_0$, 
$d (L_0, L) \leq K \cdot \# (L_0 \cap L) + c$ (the proof is given for $K = \frac{1}{8}$).

\medskip

At the same time, the quasimorphisms constructed above allow us to relate distance between diameters to the 
number of intersection points. We use the flow generated by $H_\varepsilon$ from the proof of 
Proposition~\ref{p:infinite} to obtain a lower bound for $K$.

\begin{cor}
 The optimal value for the coefficient $K$ is at least $\frac{1}{16}$.
\end{cor}
\begin{proof}
	Consider $\phi^t$ given by the time-$t$ map of the flow of $H_\varepsilon$ defined above. 
	It rotates all points of the disk $\{ r \leq 1 - \varepsilon \}$ by angle
	$\pi t$, while each rotation by $\pi$ gives rise to two intersection points.
	For a careful choice of smoothing near the boundary (which doesn't generate 
	any ``unnecessary'' intersections) we thus obtain $\# (L_0 \cap \phi^t L_0) \leq 2t + 1$. Therefore
	\begin{eqnarray*}
		d (L_0, \phi^t L_0) & \geq & \frac{|r_A (\phi^t)| - C}{1+2A} = \frac{A(1-A)\cdot t - C}{1+2A} \geq \\
		& \geq & \frac{A(1-A)}{2 (1+2A)} \left(\# (L_0 \cap \phi^t L_0) - (2C + 1) \right).
	\end{eqnarray*}
	Now note that the expression on the righthand side equals 
	$\frac{\# (L_0 \cap \phi^t L_0)}{16} + c$ for $A = 1/2$.
\end{proof}

\section{Linear bound} \label{S:bound}

Let $L', L'' \subset \Disk$ be two diameters. We say that they intersect transversely 
if the following holds: $\exists \varepsilon_{-1}, \varepsilon_1 > 0$
such that:
\begin{enumerate}
	\item $(-1, -1 + \varepsilon_{-1}] \subset L', L''$
	\item $[1 - \varepsilon_1, 1) \subset L', L''$
	\item $L'$ intersects $L''$ transversely outside $(-1, -1 + \varepsilon_{-1}] \cup [1 - \varepsilon_1, 1)$ 
	and the intesection set is compact.
\end{enumerate}
If this is the case we denote by $\#(L' \cap L'')$ the number of transverse intersection points
between $L'$ and $L''$. Obviously this number is finite. Note that a generic diameter satisfies the
conditions of transverse intersection with $L_0$.

\medskip

We now give a combinatorial description of the image of a diameter
in the reduced space $\Eqr$.

Let $L = \phi (L_0) \subset \Disk$ be a diameter intersecting $L_0$ transversely. We call a point 
$p \in \Disk \setminus (L_0 \cup L)$ \emph{northern} (\emph{southern}) if it belongs to the 
upper (lower) half of $\Disk \setminus L_0$ and \emph{white} (\emph{black}) if it is an image of a 
northern (southern) point under $\phi$. We construct a graph $G (L) = (V, E)$ (see Figure~\ref{F:treeq}) 
associated to $L$ as follows: the vertices $V = V (G)$ are the set of connected components of 
$\Disk \setminus (L_0 \cup L)$, and two vertices $v_1, \, v_2 \in V$ are adjacent if they have 
a common boundary along a segment of $L_0$. Here we omit the two segments adjacent to the boundary 
$\partial \Disk$ where $L$ coincides with $L_0$ and which separate regions of different colors.

The graph $G(L)$ carries the following additional information: to each vertex $v \in V$ 
we associate a \emph{weight} $w (v) \in (0, 1/2)$ which equals the area of the connected 
component corresponding to $v$. Next, we introduce a linear ordering of the edges of $G(L)$ 
which corresponds to their order as segments of $L_0$ (going from left to right).

\begin{figure}[!htbp]
\begin{center}
\includegraphics[width=0.7\textwidth]{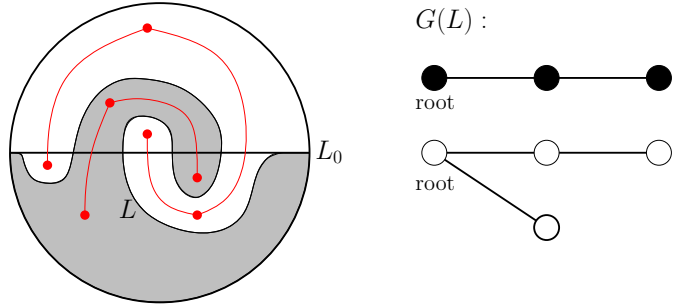}
\end{center}
\caption{Graph construction}
\label{F:treeq}
\end{figure}

We have the following properties of $G = G(L)$:
\begin{itemize}
	\item
		$G$ is a union of two \emph{trees} (one for the black vertices and another for the white ones). 
		This follows from the fact that the set of vertices of the same color is \emph{connected} 
		in $G$ (the union of black / white regions is a half-disk),
		and it \emph{does not contain loops} (otherwise the graph of the opposite color 
		will become disconnected - each loop will bound a connected component).
		
		The union of the two trees is not connected in $G$ since each edge connects vertices of 
		the same color.
	\item
		Edges that connect white / black regions correspond to alternating segments on $L_0$.
		Therefore, the numbers of black and white vertices are either equal or differ by one.
	\item
		The sum of the weights of all white/black/northern/southern regions is equal to 
		$1/2 = \Vol(\Disk, \omega) / 2$.
	\item
		There are two distinguished vertices 
		(a single black and a single white one) which correspond to the regions adjacent to the boundary 
		$\partial \Disk$, see Figure~\ref{F:treeq}. These two vertices play special role because points 
		near the boundary are stationary under compactly supported diffeomorphisms. 
		We will call these vertices the \emph{roots} of the corresponding trees.
	\item
		The total number of vertices is
		\begin{equation}
			\#V(G) = \# (L_0 \cap L) + 3
		\end{equation}
		where the counting for $\# (L_0 \cap L)$ goes only over the transverse intersections.
\end{itemize}

\medskip
The graph $G(L)$ is invariant under the action of Hamiltonians from $S$, 
therefore $G(L)$ depends only on the equivalence class $[L] \in \Eqr$. The 
following claim states that this correspondence $[L] \to G(L)$ is one-to-one.

\begin{prop}
Let $L_1, L_2$ be two diameters with isomorphic graphs, such that the graph isomorphism 
$g: G(L_1) \to G(L_2)$ preserves also weights, colors and ordering of the edges along $L_0$.
Then there exists $\phi \in S$ such that $\phi (L_1) = L_2$.
\end{prop}
\begin{proof}
	Let $L_1$, $L_2$ be such that $G(L_1) \simeq G(L_2)$. 
	Correspondence of the orderings of edges $E(L_1), E(L_2)$ implies that the intersection points 
	of $L_i \cap L_0$	appear in the same order. We apply a deformation $\phi \in S$ supported in a neighbourhood
	of $L_0$ to ensure $L_1 \cap U = L_2 \cap U$ for some neighbourhood $U$ of $L_0$. 
	Then, as corresponding regions of 
	$\Disk \setminus \{L_i \cup L_0\}$, $i = 1, 2$ have the same areas, there are no obstructions to deform each one 
	into another by a Hamiltonian supported in $\Disk \setminus L_0$. Applying these deformations 
	consequently to each of the regions we eventually deform $L_1$ to $L_2$.
\end{proof}

\begin{rem}
Not any pair of trees may come from this construction - 
there are topological/combinatorial restrictions on graphs which correspond to 
diameters in $\Disk$.
\end{rem}

We return to the proof of Theorem~\ref{T:bound}.
Applying an appropriate Hamiltonian diffeomorphism to both $L, L'$ we may assume that $L = L_0$.
In what follows we describe a way to construct a Hamiltonian flow which deforms
$L'$ to $L_0$ and whose Hofer length is bounded by $\frac{1}{8} \cdot \# L \cap L' + c$, with $c \leq 1$.
	
\begin{rem}
	The constant $\frac{1}{8}$ might not be sharp, however the previous section gives a lower bound
	of $\frac{1}{16}$ for the actual sharp constant.
\end{rem}

\begin{lm}\label{L:swapper}
	Let $\gamma_1, \gamma_2$ be two smooth embedded curves in $\Disk$ which intersect 
	transversely in two points.
	Denote by $A$ the closed domain bounded by $\gamma_1 \cup \gamma_2$, and denote by $a$ the area of $A$. 
	Let $\Gamma$ be a smooth embedded curve which connects two boundary points $p_1, p_2 \in \partial A$ 
	($p_1 \in \gamma_1$, $p_2 \in \gamma_2$) and does not touch $\gamma_1 \cup \gamma_2$ except 
	for the endpoints. 
	Let $\gamma_3, \gamma_4$ be two curves in $\Disk \setminus A$ which intersect $\Gamma$.
	Then for any small $\varepsilon > 0$ there exists a Hamiltonian diffeomorphism $\phi_\varepsilon$ of 
	norm at most $a + \varepsilon$, supported in a small neighborhood of $A \cup \Gamma$ such that
	$\phi_\varepsilon (\gamma_2) \cap \gamma_1 = \phi_\varepsilon (\gamma_2) \cap \gamma_4= \phi_\varepsilon (\gamma_3) \cap \gamma_1 = \emptyset$.
\end{lm}
\begin{proof}
	Applying a symplectomorphism, if necessary, we may approximate $A$ by a rectangle.
	Consider the flow generated by a Hamiltonian $H_\varepsilon$ as shown in Figure~\ref{F:swap}. 
	$H_\varepsilon$ is zero outside	the outer rectangle, equals $a+\varepsilon$ in the inner part, 
	linearly increases from $\frac{\varepsilon}{3}$ to $a+\frac{2\varepsilon}{3}$ on $A$ and is smooth 
	in the remaining region. It is easy to see, that for a small $\varepsilon$ and an appropriate 
	choice of smoothing, the time-1 map of this flow gives the desired Hamiltonian diffeomorphism.

\begin{figure}[!htbp]
\begin{center}
\includegraphics[width=0.55\textwidth]{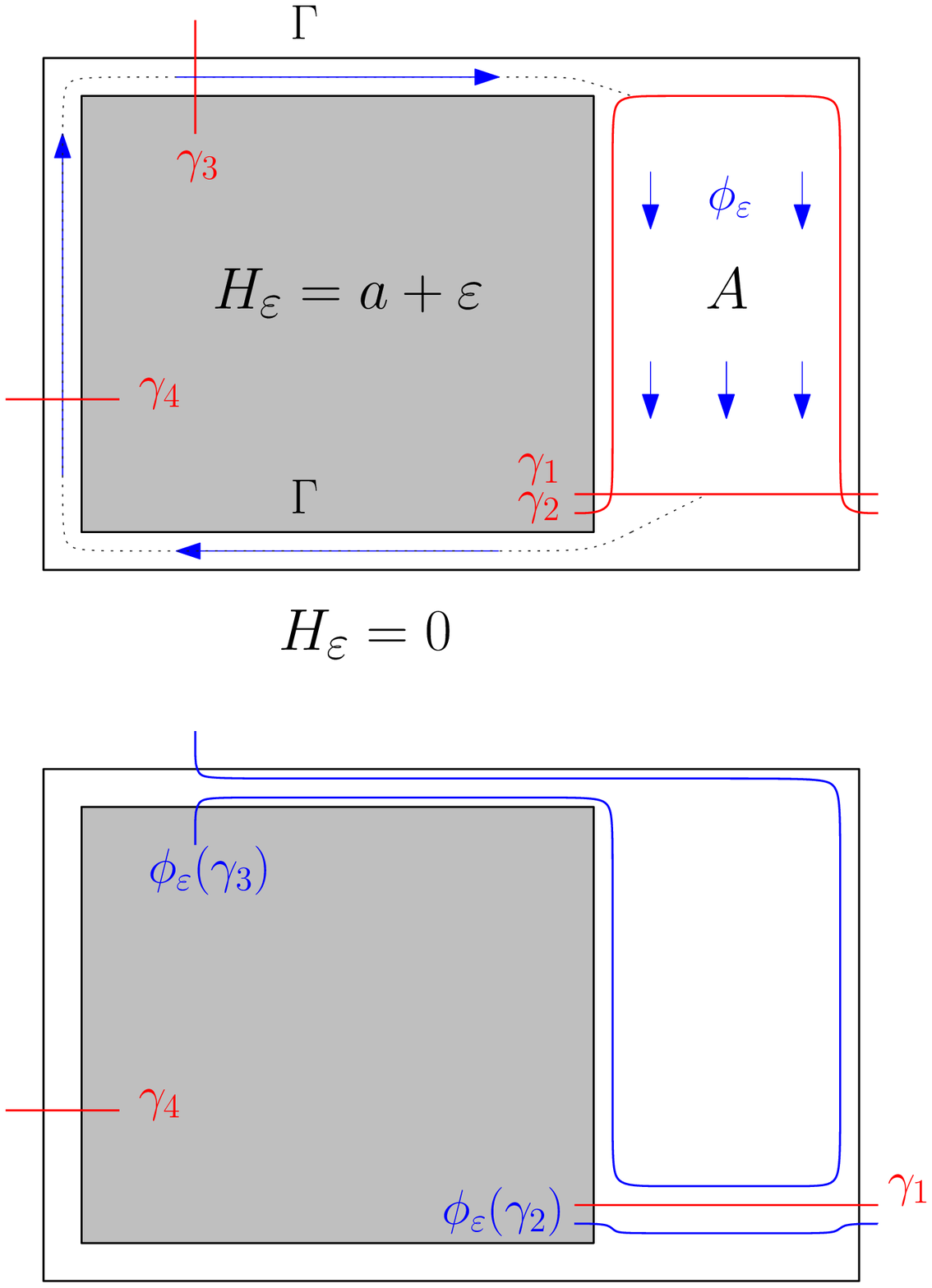}
\end{center}
\caption{}
\label{F:swap}
\end{figure}

\end{proof}

\begin{lm}\label{L:transfer}
	Let $G = G(L)$ be the graph corresponding to a diameter $L \subset \Disk$. 
	Let $v_1 \in V(G)$ be a leaf different from the root and not adjacent to it, 
	denote $a = w (v_1)$,
	$v_2$ be a vertex in $G$ with graph distance $d_G (v_1, v_2) = 2$. 
	Then for arbitrary $\varepsilon > 0$ 
	there exists a Hamiltonian diffeomorphism $\phi$ with $\| \phi \| < a + \varepsilon$ 
	which transfers weight $a$ from $v_1$ to $v_2$ and removes $v_1$ from the graph without 
	introducing new intersection points.
	(Such removal will result in merging of two vertices in the 
	tree of the opposite color.) 
	Application of such $\phi$ does not change areas of other regions by more than $\varepsilon$.
\end{lm}
\begin{proof}
	Without loss of generality we may assume that both $v_1, v_2$ are black northern vertices. 
	Let $U$ be a neighbourhood of $v_1$ in $\Disk$, denote $\gamma_1 = L_0 \cap U$, 
	$\gamma_2 = L \cap U$. The statement of the lemma implies that $\gamma_2$ is transverse to
	$\gamma_1$ in intersection points.
	Denote $\gamma_3 = L \setminus \gamma_2$, $\gamma_4 = L_0 \setminus \gamma_1$.
	Let $v \in V(G)$ be the common neighbour of $v_1, v_2$.
	Pick a curve $\Gamma$ which starts from $\gamma_1$, intersects $L_0$ going 
	through $v$ to $v_2$ along the black domain, and finally arrives to $\gamma_2$ without 
	intersecting neither $\gamma_2$ nor $L_0$ anymore (see Figure~\ref{F:treeop}).
	
	The lemma follows by application of Lemma~\ref{L:swapper}.
	The statement about $\gamma_3, \gamma_4$ ensures that $\phi$ 
	does not create additional intersection points between $L_0$ and $L$.

\begin{figure}[!htbp]
\begin{center}
\includegraphics[width=0.9\textwidth]{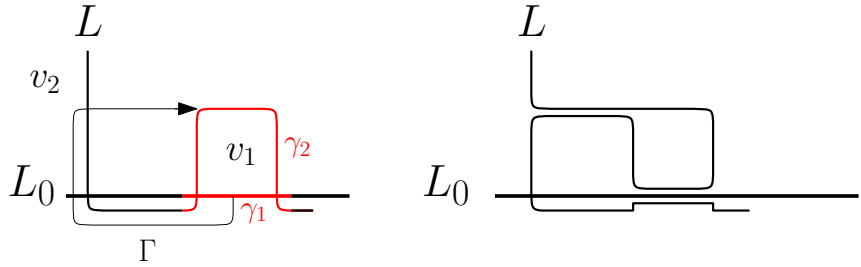}
\end{center}
\caption{Leaf deletion}
\label{F:treeop}
\end{figure}

\end{proof}

\emph{Proof of Theorem~\ref{T:bound}.}
Let $L'$ be a diameter transverse to $L_0$. Denote by $T$ the black tree of $G(L')$.
Ler $r$ be the root vertex of $T$.
Denote by $n = \# L' \cap L_0$ the number of transverse intersection points.
Note that $\depth(T)$ is bounded by $\#E(T) \leq n/2 + 1$.

We pass in turn over vertices of $T$, each time selecting a leaf, wiping it out the tree 
and transferring its weight 2 steps in direction of $r$ with the help of Lemma~\ref{L:transfer}. 
In the end we remain with all the weight concentrated at the root and its neighbours, 
with all other vertices having disappeared from the graph. 
Weight of each vertex $v \in V(T)$ is involved in at most $\frac{d_T (v, r)}{2}$ operations of this 
type, therefore the total cost of these transfers is bounded by
\[
	\sum \| \phi_j \| \leq \sum_{v_i \in T} \left( w(v_i) + \varepsilon \right) \cdot \frac{d_T(r, v_i)}{2} <
	\frac{1}{2} \cdot \frac{\depth(T)}{2} < \frac{n}{8} + \frac{1}{2} .
\]

It remains to transfer the weights from the neighbours of $r$ to the root. 
This operation can be achieved by a cost of $\sum_{v \neq r} w (v) + \varepsilon < 1/2$. This 
is so because now the graph corresponds to a diameter which has a ``simple'' intersection 
pattern with $L_0$, as depicted in Figure~\ref{F:treepush}.

\begin{figure}[!htbp]
\begin{center}
\includegraphics[width=0.3\textwidth]{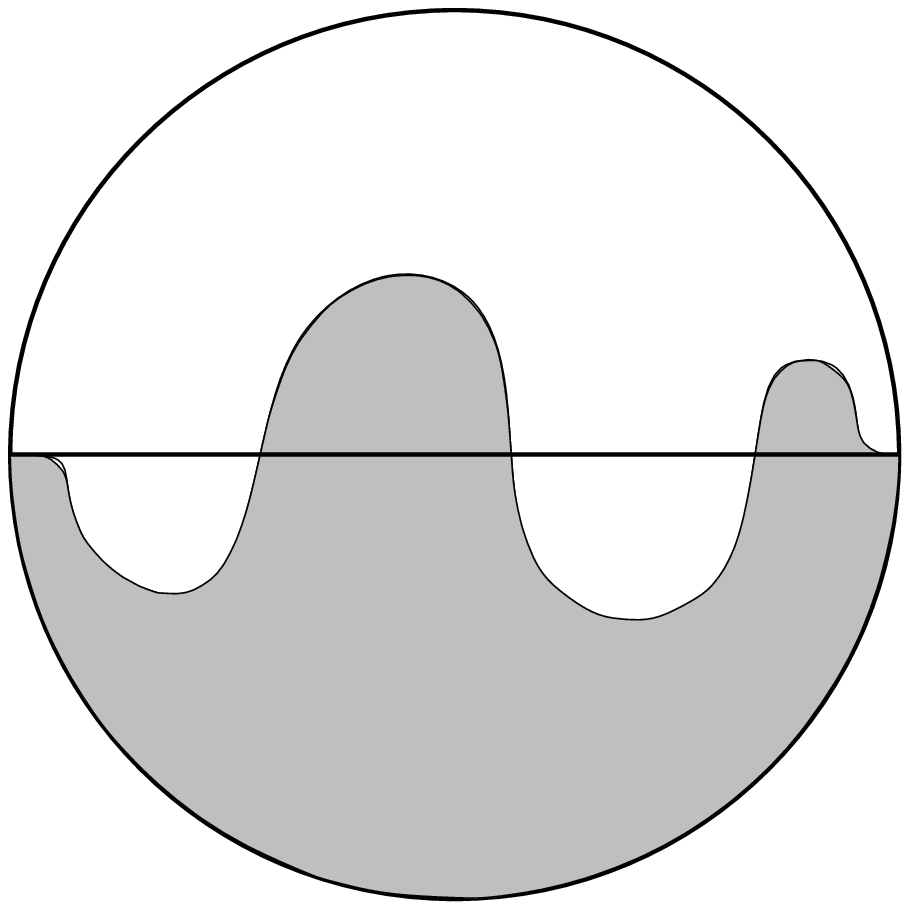}
\end{center}
\caption{}
\label{F:treepush}
\end{figure}

\Qed

\section{Generalizations} \label{S:generalize}

One may try to extend the results of this paper to other surfaces. For compact surfaces $\Sigma$ 
of positive genus the problem is quite well understood. For example, using the argument introduced in 
~\cite{L-MD} one may prove that a closed embedded smooth curve $\gamma \subset \Sigma$ with nontrivial 
$[\gamma] \in H_1 (\Sigma; \ZZ)$ can be pushed arbitrarily far from itself (in Hofer metric) using 
Hamiltonian isotopies.

In this section we state 
(without proof) few results for the cylinder $\Gamma = S^1 \times (-1, 1)$ and the sphere 
$S^2$ (both equipped with their standard symplectic structures).

We define \emph{equator} as an image of a simple smooth closed curve which
divides the manifold into two  components of equal area.
In the case of $\Gamma$ we add the requirement that the equator generates $H_1(\Gamma; \ZZ)$.
An equivalent definition is that an equator is an image of the standard equator $L_0$
($L_0 = \{z = 0\}$ for $S^2$, $L_0 = S^1 \times \{0\}$ for $\Gamma$) by a compactly 
supported Hamiltonian diffeomorphism.

\medskip

In the case of $\Gamma$ both Theorem~\ref{T:infinite} and Theorem~\ref{T:bound} remain 
true (with a different lower bound for the constant $K$). The proof is similar to
that presented above. See ~\cite{Kh-Thesis}.

\medskip

However, the case of $S^2$ is different. We do not know whether the space of equators $\Eq$ has finite 
or infinite Hofer diameter. The argument used in the proof of Theorem~\ref{T:infinite}
fails as it is not known whether there exist two different homogeneous Calabi quasimorphism on $S^2$
(there exists at least one, see ~\cite{E-P}). 
The question about finiteness of the diameter is closely related to the question of existance of 
another Calabi quasimorphism. For $S^2$ one may also show that 
$d(\phi^t (L_0), L_0) \leq c(\phi)$ (the constant is independent of $t$) for any family 
$\phi^t \in \Ham (S^2)$ generated by an autonomous flow. Therefore $L_0$ cannot be pushed 
infinitely far from itself using 1-parametric family of diffeomorphisms generated by any 
autonomous $F : S^2 \to \RR$.

A version of Theorem~\ref{T:bound} remains true. Actually, using the fact that for $S^2$ the equator 
graphs do not have distinguished roots which contain stationary points, we can improve the linear estimate to a logarithmic one:

\begin{thm} \emph{(See ~\cite{Kh-Thesis})}
	Let $L, L'$ be two equators in $S^2$ with transverse intersections. Then for any
	$\varepsilon > 0$, 
	\[
		d (L, L') \leq \log_{9/4-\varepsilon} (\# L \cap L') + c(\varepsilon).
	\]
	where $c(\varepsilon)$ is a constant which depends on $\varepsilon$ but not on the 
	equators $L, L'$.
\end{thm}


\end{document}